\documentclass[a4paper,12pt,twoside]{amsart}

\usepackage[cm]{fullpage}

\usepackage{mymacros}
\usepackage[pagebackref]{hyperref}

\usepackage{oitp}

\title[]{Rabinowitz Fukaya categories
as Cluster categories
}
\author[Y.~Lekili]{Yank\i\ Lekili}
\address{
Department of Mathematics,
Imperial College London,
South Kensington,
London,
SW7 2AZ}
\email{y.lekili@imperial.ac.uk}
\author[K.~Ueda]{Kazushi Ueda}
\address{
Graduate School of Mathematical Sciences,
The University of Tokyo,
3-8-1 Komaba,
Meguro-ku,
Tokyo,
153-8914,
Japan}
\email{kazushi@ms.u-tokyo.ac.jp}
\date{}

\usepackage{comment}

\begin{document}

\begin{abstract}
We discuss
homological mirror symmetry
for Rabinowitz Fukaya categories
of Milnor fibers
of double suspensions
of invertible polynomials,
and prove it for Brieskorn--Pham polynomials
which are not of Calabi--Yau type. 
This allows a calculation of the Rabinowitz Floer homology
of the Milnor fiber
as the Hochschild homology of the dg category
of equivariant matrix factorizations. 
\end{abstract}

\maketitle

\section{Introduction}

\subsection{}
For a pair $(k,m)$ of integers satisfying
$k \ge 2$ and $m-k \ge 2$,
let
\begin{align}
\Uv \coloneqq \lc (x,y,z,w) \in \bC^4 \relmid x^k + y^{m-k} + z^2 + w^2 = 1 \rc
\end{align}
be the Milnor fiber
of the double suspension
of the Brieskorn--Pham polynomial
$x^k+y^{m-k}$.
The cases $(k,m) = (2,r+3)$, $(3,6)$, $(3,7)$, and $(3,8)$
give simple singularities of types
$A_{r}$, $D_4$, $E_6$, and $E_8$ respectively,
for whom
homological mirror symmetry
conjectured in \cite{MR4442683}
is proved in \cite{MR4371540}.

\subsection{}

Fix an algebraically closed field $\bfk$ of characteristic zero
as coefficients for Fukaya categories.
The
\emph{stable Fukaya category}
of the Liouville manifold $\Uv$
is defined
as the quotient
\begin{align}
  \sfuk{\Uv} \coloneqq \wfuk{\Uv} / \fuk{\Uv}
\end{align}
of the wrapped Fukaya category $\wfuk{\Uv}$
by its full subcategory $\fuk{\Uv}$
consisting of compact Lagrangian submanifolds.

\subsection{}

Set
\begin{align}
R \coloneqq \bfk[x,y]/(x^k+y^{m-k}),
\end{align}
and
let
$G$ be the diagonal subgroup of $\SL_2$
isomorphic to
$
\bmu_m = \Spec \bfk[\xi]/(\xi^m-1)
$.
The stable category
$
\CMbar_G(R)
$
of
the category
$
\CM_G(R)
$
of $G$-equivariant maximal Cohen--Macaulay $R$-modules
is the homotopy category
of the dg category
$
\emf{\bA^2}{G}{x^k+y^{m-k}}
$
of $G$-equivariant matrix factorizations
of $x^k+y^{m-k}$
on $\bA^2 = \Spec \bfk[x,y]$
\cite{MR570778},
which is quasi-equivalent
to the \emph{stable derived category}
$
\scoh X \coloneqq \coh X / \perf X
$
defined as the dg quotient
of the bounded derived category of coherent sheaves
on the quotient stack
$
X \coloneqq \ld \Spec R \middle/ G \rd
$
by the full subcategory consisting of perfect complexes
\cite{Buchweitz_MCM,MR2101296}.

\subsection{}

As shown in \cite{MR3534971},
the category
$
\CM_G(R)
$
gives an additive categorification
of the cluster algebra $\bfk[\Gr_{k,m}]$.
Moreover,
the endomorphism ring of a cluster-tilting module
is described by a dimer model on a disk \cite{MR3534972},
which is originally introduced
in \cite{0609764}
to describe parametrizations
of cells in totally nonnegative Grassmannians,
and used in \cite{MR2205721}
to give a structure of a cluster algebra
on the homogeneous coordinate ring of $\Gr(k,m)$.

\subsection{}

In this paper,
we give two proofs of an equivalence
\begin{align} \label{eq:motivation}
\emf{\bA^2}{G}{x^k+y^{m-k}} \simeq \sfuk{\Uv}
\end{align}
and its generalizations.
From \pref{eq:motivation} and on,
dg categories of matrix factorizations
and Fukaya categories are completed
with respect to cones and direct summands,
so that they are idempotent-complete
stable $\infty$-categories over $\bfk$.

\subsubsection{}
One proof is based on homological mirror symmetry
for (exact symplectic Lefschetz fibrations associated with)
invertible polynomials,
which is known for Brieskorn--Pham singularities
\cite{MR2803848}.

\subsubsection{}
The other proof is based on \pref{cj:milnor},
which is homological mirror symmetry
for Milnor fibers of invertible polynomials
\cite{MR4442683}.
We prove \pref{cj:milnor}
for Brieskorn--Pham singularities
in \pref{th:hms for U}.

\subsection{}
Let
$
\RFH_*(\Uv)
$
and
$
\rfuk{\Uv}
$
be the Rabinowitz Floer homology
and
the Rabinowitz Fukaya category
of $\Uv$
respectively.
Koszul duality holds for $\Uv$
by \cite[Theorem 6.11]{MR4442683},
so that
one has
\begin{align} \label{eq:Rabinowitz Fukaya category}
 \sfuk{\Uv} \simeq \rfuk{\Uv}
\end{align}
by \cite[Corollary 1.4]{2212.14863}
(cf.~also \cite{2309.17062}),
and
\begin{align} \label{eq:Rabinowitz Floer homology}
\RFH_*(\Uv)
\simeq
\HH_*(\rfuk{\Uv})
\end{align}
by \cite[Corollary 1.7]{2212.14863}.
Hence one can compute the Rabinowitz Floer homology
as the Hochschild homology
of the dg category of equivariant matrix factorizations.

\subsection{}
This paper is organized as follows:

\subsubsection{}
In \pref{sc:suspension},
we prove \pref{th:double suspension},
which shows that a generalization of \pref{eq:motivation}
to double suspensions of invertible polynomials
follows
from \pref{cj:HMS_sing}.
\pref{cj:HMS_sing}
is homological mirror symmetry
for invertible polynomials.
\pref{th:double suspension} implies \pref{eq:motivation}
since \pref{cj:HMS_sing} is known
for Brieskorn--Pham polynomials.

\subsubsection{}
In \pref{sc:invertible},
we show that
\pref{cj:milnor}
(which is homological mirror symmetry
for Milnor fibers of invertible polynomials)
implies \pref{eq:stable hms}
(which is a `stable' version
of homological mirror symmetry
for Milnor fibers of invertible polynomials).
Then we prove \pref{th:double suspension 2}
(which shows that
a generalization of \pref{eq:motivation}
to double suspensions of invertible polynomials
follows
from \pref{cj:milnor})
using \pref{eq:stable hms}
and the Kn\"{o}rrer periodicity.

\subsubsection{Remark}
\pref{th:double suspension}
and \pref{th:double suspension 2}
have slightly different (but closely related) hypotheses
(homological mirror symmetry for $\bfw$ for the former
and
that for the Milnor fiber of $\bfW = \bfw + z^2 + w^2$
for the latter)
and the same conclusion.

\subsubsection{}
In \pref{sc:hms},
we prove \pref{cj:milnor}
for Brieskorn--Pham polynomials.
The proof is based on
\begin{itemize}
\item
a deformation-theoretic argument
going back to Seidel and Sheridan,
and
\item
the Koszul duality between the Fukaya category
and the wrapped Fukaya category.
\end{itemize}

\subsubsection{}
The `stable' homological mirror symmetry \pref{eq:stable hms}
gives an algorithm
to compute the Rabinowitz Floer homology explicitly.
We give sample calculations
in \pref{sc:RFH}.

\subsubsection*{Acknowledgments}
We thank the anonymous referee
for valuable comments and suggestions.

\section{Double suspensions of invertible polynomials} \label{sc:suspension}

A weighted homogeneous polynomial
$
 \bfw \in \bC[x_1,\ldots,x_n]
$
with an isolated critical point at the origin
is \emph{invertible}
if there is an integer matrix
$
 A = (a_{ij})_{i, j=1}^n
$
with non-zero determinant
such that
\begin{align}
 \bfw = \sum_{i=1}^n \prod_{j=1}^n x_j^{a_{ij}}.
\end{align}
The \emph{transpose} of $\bfw$ is defined in \cite{MR1214325} as
\begin{align}
 \bfwv =  \sum_{i=1}^n \prod_{j=1}^n x_j^{a_{ji}},
\end{align}
whose exponent matrix $\Av$
is the transpose matrix of $A$.
The group
\begin{align}
\Gamma \coloneqq
\lc (t_1, \ldots, t_n) \in (\Gm)^{n} \relmid
t_1^{a_{11}} \cdots t_n^{a_{1n}}
= \cdots
= t_1^{a_{n1}} \cdots t_n^{a_{nn}}
\rc
\end{align}
acts naturally on $\bA^{n}$.
The group
$
\Gammahat \coloneqq \Hom(\Gamma, \Gm)
$
of characters of $\Gamma$
is generated by
$
\chi_i \colon (t_j)_{j=1}^n \mapsto t_i
$
for
$
i = 1, \ldots, n
$
with relations
$
\chi \coloneqq
\sum_{i=1}^n a_{1i} \chi_i
= \cdots
= \sum_{i=1}^n a_{ni} \chi_i
$.
Here,
the group structure on $\Gammahat$ is written additively.

Let
$
 \emf{\bA^n}{\Gamma}{\bfw}
$
be
the idempotent completion of
the dg category of $\Gamma_{\bfw}$-equivariant matrix factorizations
of $\bfw$,
and
$\fuk{\bfwv}$ be the Fukaya--Seidel category of
(a Morsification of) $\bfwv$.

\begin{conjecture} \label{cj:HMS_sing}
For any invertible polynomial $\bfw$,
one has an equivalence
\begin{align}
 \emf{\bA^n}{\Gamma}{\bfw}
  \simeq \fuk{\bfwv}
\end{align}
of $\infty$-categories.
\end{conjecture}

\pref{cj:HMS_sing} is stated for Brieskorn--Pham singularities in three variables
in \cite{0604361},
for polynomials in three variables associated with a regular system of weights of dual type
in the sense of Saito
in \cite{MR2683215},
and
for invertible polynomials in three variables
in \cite{MR2834726}.
It
is proved
for $n=2$ in \cite{MR4325377}, and
for Sebastiani--Thom sums of polynomials of type A and D
in \cite{MR2803848,MR3030671}.
The conjecture that
$
\emf{\bA^n}{\Gamma}{\bfw}
$
has a full exceptional collection,
which is implied by \pref{cj:HMS_sing},
is stated in
\cite[Conjecture 1.4]{MR4535018},
and proved
in \cite{MR4591879}.




Let $(\dv_1,\ldots,\dv_n,\hv)$ be the sequence of positive integers
such that
$\gcd(\dv_1,\ldots,\dv_n,h) = 1$ and
\begin{align}
\bfwv(t^{\dv_1} x_1,\ldots,t^{\dv_n} x_n)
= t^{\hv} \bfwv(x_1,\ldots,x_n),
\end{align}
which is unique since $\bfwv$ is an invertible polynomial.
We say that $\bfwv$ is of Calabi--Yau type
if
\begin{align}
  \hv = \dv_1 + \cdots + \dv_n.
\end{align}

Let
\begin{align}
\Uv \coloneqq \lc (x_1,\ldots,x_n,z,w) \in \bC^{n+2} \relmid \bfwv + z^2 + w^2 = 1 \rc
\end{align}
be the Milnor fiber
of the double suspension
of $\bfwv$
and
set
\begin{align}
G \coloneqq \lc (t_1,\ldots,t_n) \in \Gamma \relmid t_1 \cdots t_n = 1 \rc.
\end{align}
We say that
$\fuk{\Uv}$ and $\wfuk{\Uv}$ are 
\emph{Koszul dual} to each other
if there exist collections
$(S_i)_{i=1}^\mu$
and
$(L_i)_{i=1}^\mu$
of objects
generating
$\fuk{\Uv}$
and
$\wfuk{\Uv}$
such that
\begin{align}
\dim_\bfk \hom^* \lb L_i, S_j \rb = \delta_{ij},
\qquad 1 \le i,j \le \mu,
\end{align}
where $\delta_{ij}$ is the Kronecker delta,
so that
the augmented endomorphism $A_\infty$-algebras
of
$
\bigoplus_{i=1}^\mu S_i
$
and
$
\bigoplus_{i=1}^\mu L_i
$
are Koszul dual to each other.
By \cite[Theorem 6.11]{MR4442683},
this assumption is satisfied if
$\bfwv$ is a Brieskorn--Pham polynomial
not of Calabi--Yau type.

\begin{theorem} \label{th:double suspension}
\pref{cj:HMS_sing}
and Koszul duality
between $\fuk{\Uv}$ and $\wfuk{\Uv}$
implies an equivalence
\begin{align} \label{eq:double suspension}
\emf{\bA^n}{G}{\bfw} \simeq \sfuk{\Uv}
\end{align}
of $\infty$-categories.
\end{theorem}

\pref{th:double suspension}
can be regarded as a
`stable'
homological mirror symmetry
for $\Uv$.
\pref{th:double suspension} implies \pref{eq:motivation}
since \pref{cj:HMS_sing}
and Koszul duality
are
known in this case.

The \emph{$n$-cluster category}
of a pretriangulated $A_\infty$-category $\scrA$
with a Serre functor $\bS$
is
defined as the orbit category
with respect to the shift
$\bS[-n]$
of the Serre functor
(see e.g.~\cite{MR3966760} and references therein).
\pref{eq:double suspension} is obtained as the composite of equivalences
\begin{align} \label{eq:CMbar=C}
\emf{\bA^n}{G}{\bfw} \simeq C_n \lb \emf{\bA^n}{\Gamma}{\bfw} \rb
\end{align}
and
\begin{align} \label{eq:sfuk=C}
C_n \lb \fuk{\bfwv} \rb
\simeq
\sfuk{\Uv}.
\end{align}
The relation between stable Fukaya categories
and cluster categories
was first pointed out in \cite{MR4442683}
and
studied further in \cite{2209.09442}.

\begin{proof}[Proof of \pref{eq:CMbar=C}]
Note that
one has an isomorphism
\begin{align}
(\chi) \simeq [2]
\end{align}
of endofunctors of
$
\emf{\bA^n}{\Gamma}{\bfw}.
$
Graded Auslander--Reiten duality
\cite{MR915178}
shows that
\begin{align}
\bS \coloneqq (\chi - \chi_1 - \cdots - \chi_n)[n-2]
\end{align}
is a Serre functor of
$
\emf{\bA^n}{\Gamma}{\bfw}
$
(see \cite[Theorem 2.5]{MR3063907}).
Hence
one has
\begin{align}
  \bS[-n] \simeq (-\chi_1-\cdots-\chi_n),
\end{align}
so that the orbit category of
$
\emf{\bA^n}{\Gamma}{\bfw}
$
with respect to
$
\bS[-n]
$
is equivalent to
the category of matrix factorizations of $\bfw$
graded by
$
\Gammahat/(-\chi_1-\cdots-\chi_n) \simeq \Ghat,
$
which is nothing but the category
$
\emf{\bA^n}{G}{\bfw}
$
of $G$-equivariant matrix factorizations of $\bfw$.
This concludes the proof of \pref{eq:CMbar=C}.
\end{proof}

\begin{proof}[Proof of \pref{eq:sfuk=C}]
Let $\scrG$ be
the $(n+1)$-Calabi--Yau completion
of
$
\scrA \coloneqq \fuk{\bfw}
$
in the sense of
\cite{MR2795754}.
Let further
\begin{align}
  \scrB \coloneqq \scrA \oplus \scrA^\dual[-n-1]
\end{align}
be 
the trivial extension algebra
of degree $n+1$
of $\scrA$,
where the $\scrA$-bimodule
\begin{align}
  \scrA^\dual \coloneqq \hom_{\bfk}(\scrA, \bfk)
\end{align}
is the graph of the Serre functor.
Then $\scrG$ is Koszul dual to $\scrB$
by \cite[Theorem 6]{MR4532824}.
It follows from \cite[Theorem 2]{MR2184464}
that
\begin{align}
  \pseu \scrB/\perf \scrB \simeq C_n(\scrA),
\end{align}
where
$\pseu \scrB$ is the category of \emph{pseudo-perfect $\scrB$-modules}
(i.e., dg modules over $\scrB$
which are perfect as $\bfk$-modules).

One has
\begin{align}
\fuk{\Uv} \simeq \perf \scrB
\end{align}
by \cite[Corollary 6.5]{MR2651908}.
Since $\fuk{\Uv}$ and $\wfuk{\Uv}$ are Koszul dual to each other,
one has
\begin{align}
\wfuk{\Uv} \simeq \perf \scrG.
\end{align}
The Koszul duality between $\scrG$ and $\scrB$ implies
\begin{align}
\pseu \scrB \simeq \perf \scrG,
\end{align}
and \pref{eq:sfuk=C} is proved.
\end{proof}

\section{More general invertible polynomials} \label{sc:invertible}

Let
\begin{align}
\Uv \coloneqq \lc (x_1,\ldots,x_{n+2}) \in \bC^{n+2} \relmid \bfWv = 1 \rc
\end{align}
be the Milnor fiber
of an invertible polynomial
\begin{align}
 \bfWv =  \sum_{i=1}^{n+2} \prod_{j=1}^{n+2} x_j^{a_{ji}}
\end{align}
in $n+2$ variables.
The group
\begin{align}
K \coloneqq
\lc (t_0, \ldots, t_{n+2}) \in (\Gm)^{n+3} \relmid
t_1^{a_{11}} \cdots t_n^{a_{1,n+2}}
= \cdots
= t_1^{a_{n+2,1}} \cdots t_n^{a_{n+2,n+2}}
= t_0 \cdots t_{n+2}
\rc
\end{align}
acts diagonally on $\bA^{n+3}$
making
$
\bfW-x_0 \cdots x_{n+2} \colon \bA^{n+3} \to \bA^1
$
equivariant,
where
\begin{align}
 \bfW =  \sum_{i=1}^{n+2} \prod_{j=1}^{n+2} x_j^{a_{ij}}
\end{align}
is the transpose of $\bfWv$.
Set
$
\Ut
\coloneqq 
\Spec
\left.
  \bfk \ld x_0,\ldots,x_{n+2} \rd
\middle/
  \lb \bfW - x_0 \cdots x_{n+2} \rb
\right.
$
and
$
U
\coloneqq
\ld \Ut \middle/ K \rd
$,
so that
\begin{align}
\emf{\bA^{n+3}}{K}{\bfW - x_0 \cdots x_{n+2}} \simeq \scoh U.
\end{align}

\begin{lemma} \label{lm:sing U}
The singular locus of $\Ut$ is the $x_0$-axis.
\end{lemma}

\begin{proof}
The singular locus of $\Ut$ is defined
inside the ambient space
$
\bA^{n+3} = \Spec \bfk[x_0,\ldots,x_{n+2}]
$
by
\begin{align}
\bfW - x_0 \cdots x_{n+2}
= x_1 \cdots x_{n+2}
= \frac{\partial \bfW}{\partial x_1} - x_0 x_2 \cdots x_{n+2}
= \cdots
= \frac{\partial \bfW}{\partial x_{n+2}} - x_0 \cdots x_{n+1}
= 0.
\end{align}
It follows from $x_1 \cdots x_{n+2}=0$ that $x_i = 0$
for some $i \in \{ 1, \ldots, n+2 \}$,
and one may assume $i=1$ without loss of generality.
Then one has
\begin{align} \label{eq:partial 2 to n+2}
\frac{\partial \bfW}{\partial x_2}
= \cdots
= \frac{\partial \bfW}{\partial x_{n+2}}
= 0,
\end{align}
and one can show
using the classification of invertible polynomials
\cite{MR1188500}
that \pref{eq:partial 2 to n+2} implies
$x_j = 0$ for some $j \in \{ 2, \ldots, n+2 \}$.
Then one has
\begin{align}
\frac{\partial \bfW}{\partial x_1} = x_0 x_2 \cdots x_{n+2}
= 0,
\end{align}
which together with \pref{eq:partial 2 to n+2}
implies $x_1=\cdots=x_{n+2}=0$
since $\bfW$ is an invertible polynomial.
\end{proof}

If $\bfW$ is not of Calabi--Yau type,
then any point of
$
\Ut
$
satisfying $x_0 \ne 0$
can be brought to $x_0 = 1$ by the action of $K$,
so that
the complement
$
D \coloneqq U \setminus E
$
of the closed substack $E$ of $U$
defined by
$
x_0 = 0
$
can be identified with
$
\ld \Dt \middle/ H \rd
$
where
\begin{align}
\Dt
\coloneqq
\Spec \bfk[x_1,\ldots,x_{n+2}]/(\bfW - x_1 \cdots x_{n+2})
\end{align}
and
\begin{align}
H
&\coloneqq \lc (t_1,\ldots,t_{n+2}) \in (\Gm)^{n+2}
\relmid
t_1^{a_{11}} \cdots t_n^{a_{1,n+2}}
= \cdots
= t_1^{a_{n+2,1}} \cdots t_n^{a_{n+2,n+2}}
= t_1 \cdots t_{n+2} \rc
\label{eq:H} \\
&\cong K \cap \lc t_0 = 1 \rc.
\end{align}
Since the intersection of the singular locus of $U$
with $E$ is the origin,
the full subcategory
$
\scoh_E U
$
of
$
\scoh U \coloneqq \coh U / \perf U
$
consisting of objects supported on $E$
is equivalent
to the full subcategory
$\scoh_0 U$
consisting of objects supported at the origin;
\begin{align} \label{eq:scoh_0}
  \scoh_0 U \simeq \scoh_E U.
\end{align}

\begin{conjecture} \label{cj:milnor}
For any invertible polynomial $\bfW$
not of Calabi--Yau type,
one has equivalences
\begin{align} \label{eq:wfuk}
\scoh U
\simeq
\wfuk{\Uv}
\end{align}
and
\begin{align} \label{eq:fuk}
\scoh_0 U
\simeq
\fuk{\Uv}
\end{align}
of $\infty$-categories.
\end{conjecture}

\pref{eq:wfuk} is \cite[Conjecture 1.4]{MR4442683},
from which \pref{eq:fuk} should follow
as the restriction
to the full subcategories
consisting of objects $X$
such that $\hom(X,Y)$ is perfect as a $\bfk$-module
for any $Y$.
\cite[Theorem 1.1]{MR4404801}
gives \pref{eq:fuk}
for $n=0$,
and a $\bZ/2\bZ$-graded variant of \pref{eq:wfuk}
is discussed in \cite{MR4713718}.

\begin{theorem} \label{th:double suspension 2}
If \pref{cj:milnor} holds
for the double suspension
of an invertible polynomial $\bfw$,
then one has
an equivalence
\begin{align} \label{eq:double suspension 2}
\emf{\bA^n}{G}{\bfw} \simeq \sfuk{\Uv}
\end{align}
of $\infty$-categories.
\end{theorem}

\begin{proof}
The equivalence
\begin{align}
  \coh D \simeq \coh U / \coh_E U
\end{align}
induces an equivalence
\begin{align}
  \scoh D \simeq \scoh U / \scoh_E U,
\end{align}
which together with \pref{eq:scoh_0} gives an equivalence
\begin{align} \label{eq:scoh D}
  \scoh D \simeq \scoh U / \scoh_0 U.
\end{align}
It follows that
if $\bfW$ is an invertible polynomial
not of Calabi--Yau type
such that \pref{cj:milnor} holds,
then one has an equivalence
\begin{align} \label{eq:stable hms}
\scoh D \simeq \sfuk{\Uv}
\end{align}
of $\infty$-categories.

If $\bfW$ is not of Calabi--Yau type,
then the singular locus of $\Dt$ is the origin.
Indeed,
the singular locus of $\Dt$ is defined
inside the ambient space
$
\bA^{n+2} = \Spec \bfk[x_1,\ldots,x_{n+2}]
$
by
\begin{align}
\bfW - x_1 \cdots x_{n+2}
= \frac{\partial \bfW}{\partial x_1} - x_2 \cdots x_{n+2}
= \cdots
= \frac{\partial \bfW}{\partial x_{n+2}} - x_1 \cdots x_{n+1}
= 0.
\end{align}
By multiplying $d_i x_i$ and summing over $i$,
one obtains
\begin{align} \label{eq:non-weighted-homogeneous}
 h \bfW - \sum_{i=1}^{n+2} d_i x_1 \cdots x_{n+2} = 0,
\end{align}
which together with
$
\bfW - x_1 \cdots x_{n+2} = 0
$
implies
$
\bfW = x_1 \cdots x_{n+2} = 0
$.
Now we can argue
as in the proof of \pref{lm:sing U}
that $x_1=\cdots=x_{n+2}=0$.

The double suspension
$
\bfw(x_1,\ldots,x_n) + x_{n+1}^2 + x_{n+2}^2
$
is not of Calabi--Yau type
since
$
2 d_{n+1} = 2 d_{n+2} = h
$
and hence
$
\sum_{i=1}^{n+2} d_i > h
$.

Set
$
D' \coloneqq \ld \Dt' \middle/ H \rd
$
where
$
\Dt'
\coloneqq
\Spec \left. \bfk[x_1,\ldots,x_{n+2}] \middle/ (\bfW) \right.
$.
The singular loci of both $\Dt$ and $\Dt'$ are the origin,
and the formal completions of $\Dt$ and $\Dt'$ at the origin
are isomorphic since
\begin{multline}
  \bfw(x_1,\ldots,x_{n}) + x_{n+1}^2 + x_{n+2}^2 - x_1 \cdots x_{n+2} \\
  =\bfw(x_1,\ldots,x_{n})
  +\left( \sqrt{1-\frac{1}{4} (x_1 \cdots x_{n})^2} x_{n+1} \right)^2
  +\left( x_{n+2}-\frac{1}{2} x_1 \cdots x_{n+1} \right)^2
\end{multline}
in
$
\bfk \llbracket x_1,\ldots, x_{n+2} \rrbracket
$.
It follows that
\begin{align}
\scoh D
\simeq
\scoh D'
\end{align}
by \cite[Theorem 2.10]{MR2735755}.
The isomorphism
\begin{align}
H \cong G \times \bmu_2 \times \bmu_2
\end{align}
and the Kn\"{o}rrer periodicity
\cite[Proposition 2.1]{MR877010}
shows
\begin{align}
\scoh D'
\simeq
\emf{\bA^n}{G}{\bfw},
\end{align}
and \pref{th:double suspension} is proved.
\end{proof}

\section{Homological mirror symmetry for Milnor fibers of Brieskorn--Pham singularities} \label{sc:hms}

We use the same notations as in \pref{sc:invertible}.
We prove the following theorem in this section:

\begin{theorem} \label{th:hms for U}
\pref{cj:milnor} holds
if $\bfW$ is a Brieskorn--Pham polynomial.
\end{theorem}

Let $M$ be the free abelian group generated by
$
\{ \bfe_i \}_{i=1}^{n+2},
$
and
$\Mt$ be the subgroup
generated by
$
\lc
\bff_j
\coloneqq
\sum_{i=1}^{n+2} a_{ji} \bfe_i
\rc_{i=1}^{n+2}.
$
We write the inclusion $\Mt \hookrightarrow M$ as $\varphi$.

Let
\begin{align}
\cP \coloneqq \lc (y_1,\ldots,y_{n+2}) \in M_{\bCx} \coloneqq M \otimes \bCx
\relmid y_1 + \cdots + y_{n+2} + 1 = 0 \rc
\end{align}
be an $n$-dimensional pair of pants,
and
set
$
\cPt \coloneqq \varphi_{\bCx}^{-1} (\cP)
$
where
\begin{align}
\varphi_{\bCx} \coloneqq \varphi \otimes \bCx \colon \Mt_{\bCx} \to M_{\bCx}, \quad
\lb x_i \rb_{i=1}^{n+2} \mapsto
\lb y_i = \prod_{j=1}^{n+2} x_j^{a_{ji}} \rb_{i=1}^{n+2}.
\end{align}
The closure of $\cPt$ in
$
\Mt_{\bC}
$
is identified with the Milnor fiber
\begin{align}
\Uv = \lc (x_1,\ldots,x_{n+2}) \in \bC^{n+2} \relmid \bfWv(x_1,\ldots,x_{n+2})+1=0 \rc
\end{align}
of $\bfWv$.

We equip $\Uv$ with the grading
defined by the tensor square $\Omega_\Uv^{\otimes 2}$
of the holomorphic volume form
\begin{align}
 \Omega_\Uv
  \coloneqq \Res \frac{d x_1 \wedge \cdots \wedge d x_{n+2}}{\bfWv(x_1,\ldots,x_{n+2})+1}.
\end{align}
If $n \ge 1$,
then
the choice of a grading of $\Uv$ is unique
because of the simple connectivity of $\Uv$.

The divisors
\begin{align}
E_i = \lc (x_1,\ldots,x_{n+2}) \in \Uv \relmid x_i = 0 \rc, \quad
i = 1, \ldots, n+2
\end{align}
are smooth,
and one has
$
 \Uv \setminus \cPt = \bigcup_{i=1}^{n+2} E_i.
$
If $n \ge 1$,
then $E_i$ is non-empty and connected
for any $i=1,\ldots,n+2$.

Let $\cPua$ be the universal abelian cover of $\cP$,
which agrees with the universal cover
if $n \ge 1$.
The inclusion
$
\cP \hookrightarrow M_{\bCx}
$
induces
an isomorphism
$
H_1(\cP) \simto H_1 \lb M_{\bCx} \rb \cong M,
$
so that $M$ is identified
with the group
$\Deck \lb \cPua \to \cP \rb$
of deck transformations.
The group
$
\Deck \lb \cPua \to \cPt \rb
$
is naturally identified with $\Mt$,
so that
$
\Mbar \coloneqq \operatorname{Deck} \lb \cPt \to \cP \rb
$
is identified with
$
M / \Mt.
$

The Lagrangian immersion
$
\LSS
$
from an $n$-sphere
to $\cP$
introduced by Seidel and Sheridan
\cite{MR2819674,MR2863919}
lifts to a Lagrangian immersion
$\LSSt$
from the disjoint union
of $|\Mbar|$ copies of spheres
to $\cPt$.
Let
$
\fukzero {U}
$
be the full subcategory of $\fuk U$
split-generated by $\LSSt$.

Let
\begin{align}
  \bfV(z_0,\ldots,z_{n+2}) \coloneqq \bfW(z_1,\ldots,z_{n+2}) - \prod_{i=0}^{n+2} z_i
\end{align}
be a semi-invariant element of $\bfk[z_0,\ldots,z_{n+2}]$
with respect the natural action of
\begin{align}
K \coloneqq \lc (t_0,\ldots,t_{n+2}) \in (\Gm)^{n+2} \relmid
\prod_{j=1}^{n+2} t_j^{a_{ij}} = \prod_{j=0}^{n+2} t_j
\text{ for any } i \in \{ 1, \ldots, n+2\} \rc,
\end{align}
and
$
\emfzero{\bA^{n+2}}{K}{\bfV}
$
be the full subcategory of
the dg category
$
\emf{\bA^{n+2}}{K}{\bfV}
$
of $K$-equivariant matrix factorizations of $\bfV$
split-generated by the structure sheaf of the origin.

\begin{proposition} \label{pr:fuk_0=mf_0}
One has an equivalence
$
\fukzero{U}
\simeq
\emfzero{\bA^{n+2}}{K}{\bfV}
$.
\end{proposition}

\begin{proof}
Let
$
 \LGr(\cP)^{\ua} \to \LGr(\cP)
$
be the universal abelian cover of $\LGr(\cP)$,
whose group of deck transformations
can be identified with
$
\bG
\coloneqq
H_1(\LGr(\cP)).
$
A \emph{$\bG$-grading}
of a Lagrangian $L$ in $\cP$
is a lift $\stilde$
of the tautological section
$s_L \colon L \to \LGr(\cP)$
to $\LGr(\cP)^\ua$.
The Floer cohomology of $\bG$-graded Lagrangians is
$\bG$-graded.

The tensor square
$
 \Omega_\cP^{\otimes 2}
$
of the holomorphic volume form
\begin{align}
 \Omega_{\cP}
  \coloneqq \Res \frac{1}{1+y_1+\cdots+y_{n+2}}
   \frac{d y_1}{y_1} \wedge \cdots \wedge \frac{d y_{n+2}}{y_{n+2}}
\end{align}
induces a splitting of the exact sequence
\begin{align}
 0 \to H_1(\LGr(T_p \cP)) \to H_1(\LGr(\cP)) \to H_1(\cP) \to 0,
\end{align}
where
$\LGr(T \cP)$ is the Lagrangian Grassmannian bundle
of the tangent bundle $T \cP$,
and
$p \in \cP$
is
an arbitrarily chosen base point.
We identify
$
\bG
\coloneqq
H_1(\LGr(\cP))
$
with
$
H_1(\cP) \oplus H_1(\LGr(T_p \cP))
\cong M \oplus \bZ
$
by this splitting.
Similarly,
we identify
$
\Deck \lb \LGr(\cP)^{\ua} \to \LGr \lb \cPt \rb \rb
$
with $\Mt \oplus \bZ$
using the tensor square
(of the restriction to $\cPt$)
of $\Omega_U$.

The cohomology algebra $A$
of the endomorphism $A_\infty$-algebra $\cA$
of the immersed Lagrangian sphere $\LSS$
in the $\bG$-graded Fukaya category of $\cP$
is computed
in \cite{MR2863919}
as the exterior algebra
generated by elements
$\theta_0, \ldots, \theta_{n+2}$
of degrees
\begin{align} \label{eq:deg theta_i}
 \deg \theta_i =
\begin{cases}
- (\bfe_1 + \cdots + \bfe_{n+2}) - 1 & i = 0, \\
1 + \bfe_i & i =1,\ldots,n+2.
\end{cases}
\end{align}

The Kontsevich formality
lifts the Hochschild--Kostant--Rosenberg isomorphism
to an $L_\infty$-quasi-isomorphism
\begin{align} \label{eq:hkr1}
 \PhiHKR \colon \CC^\bullet(A)
  \to \bC[z_0,\ldots,z_{n+2}][\theta_0,\ldots,\theta_{n+2}]
\end{align}
from the Hochschild cochain complex of $A$
to the graded Lie algebra
of polyvector fields on
$\bC[z_0,\ldots,z_{n+2}]$.
The latter is a formal $L_\infty$-algebra,
whose underlying graded vector space
is the the free graded commutative algebra
generated by even variables
$z_i$
of degrees
\begin{align} \label{eq:deg z_i}
 \deg z_i = - \deg \theta_i + 1
\end{align}
and odd variables $\theta_i$
of degrees \pref{eq:deg theta_i}.
The non-trivial $L_\infty$-operation $\frakl_2$
is identified
with the Schouten bracket
by sending $\theta_i$ to $\frac{\partial}{\partial z_i}$.
As explained in \cite[Lemma 3.3]{MR4243021},
the Maurer-Cartan element $\mu^{\ge 3}$
describing the deformation of $A$ to $\cA$
is sent to
\begin{align}
\bfV_0 \coloneqq - z_0 \cdots z_{n+2}
\end{align}
by $\PhiHKR$.
As shown in \cite[Lemma 3.5]{MR4243021},
this implies that the Hochschild cohomology of $\cA$
is isomorphic to the quotient
\begin{align}
  \bC[z_0,\ldots,z_{n+2}][u_0,\ldots,u_{n+2}]/\cJ
\end{align}
of the free graded commutative algebra
generated by even variables $z_i$ of degrees \pref{eq:deg z_i}
and odd variables $u_i$ of degrees $1$
by the ideal
\begin{align}
 \cJ = \lb \prod_{i \nin I} z_i \prod_{i \in I} u_i \rb_{I \subset \{ 0, \ldots, n+2 \}}.
\end{align}

Let $\cAt$ be the endomorphism $A_\infty$-algebra
of $\LSSt$
in the $\bGt$-graded Fukaya category of $\cPt$.
It is obtained from $\cA$ as follows:
\begin{itemize}
\item
Take a complete set
$
\{ m_g \}_{g \in \Mbar}
\subset M
\subset \bG \cong M \oplus \bZ
$
of representatives of $\Mbar \cong M/\Mt$.
\item
Take the $\bG$-graded endomorphism $A_\infty$-algebra
of the direct sum
$
\bigoplus_{g \in \Mbar} \cA(m_g)
$
of shifted free modules
in $\module \cA$.
\item
Take the $A_\infty$-subalgebra
consisting of homogeneous elements whose degrees are in $\bGt$.
\end{itemize}

Let
$
\cAt_R
$
be the $\bGt$-graded deformation of
$
\cAt
$
over the polynomial ring
\begin{align}
R \coloneqq \bC \ld r_1,\ldots,r_{n+2} \rd
\end{align}
whose $A_\infty$-operations are given by
counting pseudo-holomorphic disks in $U$
weighted by intersection numbers with $E_i$.
The degree
\begin{align}
 \deg r_i = \bff_i + 2 = \sum_{j=1}^{n+2} a_{ij} \bfe_j + 2
\end{align}
of the variable counting intersection numbers with $E_i$
is defined by first choosing a small disk in $U$
intersecting simply and transversely to $E_i$
and disjoint from all the other $E_j$'s,
lifting it to $\LGr(U)$,
and taking the class of its boundary
in $\bGt$.
Let
$
\frakm = (r_1, \ldots, r_{n+2})
$
be the maximal ideal of
$
R
$
at the origin.
The first order deformation class of
$
\cAt_R
$
belongs to the $\Mbar$-invariant degree 2 part
$
\HH^2 \lb \cAt, \cAt \otimes \frakm/\frakm^2 \rb^\Mbar
$
of the $\bGt$-graded Hochschild cohomology
of the $\cAt$-bimodule
$\cAt \otimes \frakm/\frakm^2$,
which is isomorphic to the degree 2 part
$
\HH^2 \lb \cA, \cA \otimes \frakm/\frakm^2 \rb
$
of the $\bG$-graded Hochschild cohomology
of the $\cA$-bimodule
$\cA \otimes \frakm/\frakm^2$
by \cite[Remark 2.66]{MR3294958}.

One has
\begin{align}
 &\deg \lb \prod_{i=1}^{n+2} r_i^{b_i} \prod_{i=0}^{n+2} z_i^{c_i} \prod_{i \in I} u_i \rb \\
 &\quad = \sum_{i=1}^{n+2} b_i \lb \sum_{j=1}^{n+2} a_{ij} \bfe_j + 2 \rb
  - \sum_{i=1}^{n+2} c_i \bfe_i
  + c_0 (\bfe_1 + \cdots + \bfe_{n+2} + 2)
  + |I| \\
 &\quad = \sum_{i=1}^{n+2} \lb \sum_{j=1}^{n+2} a_{ji} b_j - c_i + c_0 \rb \bfe_i
  + \sum_{i=1}^{n+2} 2 b_i
  + 2 c_0
  + |I|.
\end{align}
For this to be 2, one needs
\begin{align}
 c_i = \sum_{j=1}^{n+2} a_{ji} b_j + c_0
\end{align}
for $i =1,\ldots,n+2$ and
\begin{align}
 2 \sum_{i=1}^{n+2} b_i + 2 c_0 + |I| = 2.
\end{align}
This is the case
if and only if one of $b_1, \ldots, b_n$, $c_0$, or $|I|/2$ is $1$
and others are $0$.
If $b_j = \delta_{ij}$ for $i \in \{ 1,\ldots,n+2 \}$,
then one has
\begin{align}
 c_j = \sum_{k=1}^{n+2} a_{kj} \delta_{ki} = a_{ij}
\end{align}
and
\begin{align} \label{eq:first ord def classes}
 r_i \prod_{j=1}^{n+2} z_j^{c_j}
  = r_i \prod_{j=1}^{n+2} z_j^{a_{ij}}.
\end{align}
This argument also shows
$
 \HH^2(\cA, \cA \otimes \frakm^i) = 0
$
for $i \ge 2$.

The first order deformation class of $\cAt_R$
is the sum
\begin{align} \label{eq:first ord def class}
\sum_{i=1}^{n+2} r_{i} \prod_{j=1}^{n+2} z_j^{a_{ij}}
\end{align}
of \pref{eq:first ord def classes}
for all $i \in \{ 1, \ldots, n+2 \}$.
If $n > 0$,
this follows from \cite[Lemma 4.22]{MR4153652}
and \cite[Assumption 5.3]{MR4153652}.\footnote{
Although the proof of \cite[Assumption 5.3]{MR4153652} is relegated to future work,
there is no difficulty in the present setting,
where $U$ is exact.}
If $n=0$,
then one has either
\begin{enumerate}
\item
Brieskorn--Pham: $\bfW = \bfWv = x^p + y^q$,
\item
chain: $\bfW = x^p y + y^q$ and $\bfWv = x^p + x y^q$, or
\item
loop: $\bfW = \bfWv = x^p y + x y^q$.
\end{enumerate}
In the Brieskorn--Pham case,
\cite[Lemma 4.22]{MR4153652}
and \cite[Assumption 5.3]{MR4153652}
shows that
the first order deformation class of $\cAt_R$
is given by \eqref{eq:first ord def class}
just as in the case $n > 0$.
In the chain case,
since the divisor in $\Uv$ defined by $x=0$ is empty,
\cite[Lemma 4.22]{MR4153652}
and \cite[Assumption 5.3]{MR4153652}
shows that
the first order deformation class of $\cAt_R$
is given by
$
r_2 y^q
$,
which is equivalent to \pref{eq:first ord def class}
since
\begin{align}
r_1 x^p y + r_2 y^q - x y z = r_2 y^q - x y (z - r_1 x^{p-1}),
\end{align}
so that the term $r_1 x^p y$ can be absorbed
into a coordinate change of $z$.
Similarly,
in the loop case,
both of the divisors in $\Uv$ defined by $x=0$ and $y=0$ are empty,
and one has
\begin{align}
r_1 x^p y + r_2 x y^q - x y z = x y (z - r_1 x^{p-1} - r_2 y^{q-1}),
\end{align}
so that both of the terms $r_1 x^p y$ and $r_2 x y^q$ can be absorbed
into a coordinate change of $z$.

Now $\cAt_R$ and hence
\begin{align}
\cAt_1 \coloneqq \cAt_R \otimes_R R/(r_i-1)_{i=1}^n
\end{align}
are determined uniquely up to quasi-isomorphism
by \cite[Proposition 6.6]{MR3578916},
so that
\begin{align}
\fukzero U \simeq \module \cAt_1.
\end{align}

Similarly,
$
\mf_0 \lb \bfk[z_0,\ldots,z_{n+2}], \bfV_0 \rb
$
admits a $\bG$-grading by \pref{eq:deg z_i},
and
\begin{align}
\mf_0 \lb R[z_0,\ldots,z_{n+2}],
\bfV_R \coloneqq \bfV_0 + \sum_{i=1}^{n+2} r_{i} \prod_{j=1}^{n+2} z_j^{a_{ij}}
\rb
\end{align}
is a $\bGt$-graded deformation
whose first order deformation class
is \pref{eq:first ord def class}
by \cite[Proposition 7.1]{MR3294958}.
It follows that
\begin{align}
\emfzero{\bA^{n+2}}{K}{\bfV}
\simeq
\module \cAt_1,
\end{align}
and
\pref{pr:fuk_0=mf_0} is proved.
\end{proof}

Now we specialize to Brieskorn--Pham singularities
where $a_{ij} = p_i \delta_{ij}$ and $p_i > 2$
for $i \in \{ 1, \ldots, n+2 \}$.

\begin{proposition} \label{pr:fuk_0 U=fuk U}
One has an equivalence
\begin{align}
  \fukzero \Uv \simeq \fuk \Uv.
\end{align}
\end{proposition}

\begin{proof}
Let
\begin{align}
\varpi \colon \Uv \to \bC, \qquad
(x_k)_{k=1}^{n+2} \mapsto x_{n+2}
\end{align}
be the projection to the last coordinate.
The fiber
\begin{align}
\varpi^{-1} (x_{n+2})
= \lc
(x_1, \ldots, x_n) \in \bC^n
\relmid
x_1^{p_1} + \cdots + x_n^{p_n} + (x_{n+2}^{p_{n+2}} + 1) = 0
\rc
\end{align}
is a Milnor fiber of
a lower-dimensional Brieskorn--Pham singularity,
unless $x_{n+2}$ belongs to the set
\begin{align}
\Critv \varpi
=
\lc \zeta_k \coloneqq \exp \lb 2 (k+1/2) \pi \sqrt{-1} \middle/ p_{n+2} \rb \rc_{k=0}^{p_{n+2}-1}
\end{align}
of critical values of $\varpi$,
in which case the fiber has
the lower-dimensional Brieskorn--Pham singularity at the origin.

Set
\begin{align}
I \coloneqq
\lc \bsi = (i_1, \ldots, i_{n+2}) \in \bZ^{n+2} \relmid
0 \le i_k \le p_k-2 \text{ for any } 1 \le k \le n+2 \rc.
\end{align}
A distinguished basis
$
\lb V_{\bsi} \rb_{\bsi \in I}
$
of vanishing cycles
is described inductively
in \cite{MR2803848}
as a fibration of lower-dimensional vanishing cycles
above the matching path on the $x_{n+2}$-plane
connecting
$\zeta_{i_{n+2}}$ and $\zeta_{i_{n+2}+1}$.

There is another distinguished basis
$
(S_{\bsi})_{\bsi \in I}
$
of vanishing cycles,
which can similarly be constructed inductively
as a fibration of lower-dimensional vanishing cycles
above the matching path on the $x_{n+2}$-plane
connecting
$\zeta_0$ and $\zeta_{i_{n+2}+1}$.
Note that $S_\bsi=V_\bsi$
for $\bsi = \bszero \coloneqq (0, \ldots, 0)$.
The collection
$
(V_\bsi)_{\bsi \in I}
$
in the Fukaya--Seidel category
of the Brieskorn--Pham polynomial
corresponds to simple modules
of the tensor product
of the path algebras of $A_{p_k-1}$-quivers,
whereas
the collection
$
(S_\bsi)_{\bsi \in I}
$
corresponds to projective modules.

The collection
$
(S_\bsi)_{\bsi \in I}
$
of objects in $\fuk \Uv$
has a Koszul dual collection
$
(L_\bsi)_{\bsi \in I}
$
of objects in $\wfuk \Uv$,
which are connected components
of the inverse image of
\begin{align} \label{eq:positive real locus}
\lc (x_1, \ldots, x_{n+2}) \in \cP \cap \bR^{n+2} \relmid
x_i > 0 \text{ for } 1 \le i \le n \text{ and }
x_{n+2} < -1 \rc
\end{align}
by the covering map
$
\varphi_{\bCx} \colon \Uv \to \cP.
$
The Lagrangian $L_\bsi$ is a Lefschetz thimble
for $\varpi$
over the half line on the $x_{n+2}$-planes
from $\zeta_{i_{n+2}+1}$ to infinity
associated with the $(n-1)$-dimensional vanishing cycle $V_\bsibar$
where $\bsibar = (i_k)_{k=1}^n$.

Since $\LSS$ intersects \pref{eq:positive real locus}
only at one point \cite[Corollary 2.9]{MR2863919},
there is an irreducible component $\LSSt_\bszero$ of $\LSSt$
such that
$
\LSSt_\bszero \cap L_\bsi
$
consists of one point
if $\bsi = \bszero$,
and is empty otherwise.
It follows that $\LSSt_\bszero \simeq S_\bszero$ in $\fuk \Uv$.
The vanishing cycle $V_\bsi$ for other $\bsi$
is the image of $V_\bszero = S_\bszero$
by the map
$
(x_k)_{k=1}^{n+2} \mapsto
\lb \exp \lb 2 i_k \pi \sqrt{-1} \middle/ p_k \rb x_k \rb_{k=1}^{n+2}.
$
Since vanishing cycles generate $\fuk \Uv$,
\pref{pr:fuk_0 U=fuk U} is proved.
\end{proof}

The closed immersion
\begin{align}
  \iota \colon 
  \ld \bfW^{-1}(0) \middle/ K \rd
  \to
  \ld \bfV^{-1}(0) \middle/ K \rd
\end{align}
induces a push-forward
\begin{align}
  \iota_* \colon
  \scoh \ld \bfW^{-1}(0) \middle/ K \rd
  \to
  \scoh \ld \bfV^{-1}(0) \middle/ K \rd
\end{align}
since the push-forward of the structure sheaf of $\bfW^{-1}(0)$
is defined by the non-zero-divisor $x_0$
and hence has projective dimension one.
The adjunction with the pull-back
\begin{align}
  \iota^* \colon 
  \scoh \ld \bfW^{-1}(0) \middle/ K \rd
  \to
  \scoh \ld \bfV^{-1}(0) \middle/ K \rd
\end{align}
follows from that for coherent sheaves.
Let
\begin{align}
\cL_\bsi \coloneqq \cO_{\bA^1 \times \bszero} \lb -i_1 \chi_1 - \cdots - i_{n+2} \chi_{n+2} \rb
\end{align}
be the structure sheaf of the closed subscheme
$\bA^1 \times \bszero$
of $\bfV^{-1}(0)$
defined by $x_1=\cdots=x_{n+2}=0$,
twisted by an element of
$
\Khat \coloneqq \Hom(K, \Gm)
$
which is generated by
\begin{align}
\chi_i \colon K \to \Gm, \quad (\alpha_i)_{i=0}^{n+2} \mapsto \alpha_i^{-1}
\end{align}
with relations
\begin{align}
p_1 \chi_1 = \cdots = p_{n+2} \chi_{n+2} = \chi_0 + \cdots + \chi_{n+2}.
\end{align}
Then the sequence
$
\lb \cE_\bsi \coloneqq \iota^* \cL_\bsi \rb_{\bsi \in I}
$
is the full strong exceptional collection
given in \cite{MR2803848}.
Let
$
\lb \cF_\bsi \rb_{\bsi \in I}
$
be the full exceptional collection
right dual to
$
\lb \cE_\bsi \rb_{\bsi \in I}
$
so that
\begin{align}
\dim \hom \lb \cE_\bsi, \cE_{\bsj} \rb = \delta_{\bsi, \bsj},
\end{align}
and set
\begin{align}
  \cS_\bsi \coloneqq \iota_* \cF_\bsi.
\end{align}
Then
$
\lb \cL_\bsi \rb_{\bsi \in I}
$
and
$
\lb \cS_\bsi \rb_{\bsi \in I}
$
are Koszul dual;
\begin{align} \label{eq:Koszul duality}
  \dim \hom(\cL_\bsi, \cS_{\bsj})
  &\cong \dim \hom(\cE_\bsi, \cF_{\bsj})
  \simeq \delta_{\bsi, \bsj}.
\end{align}

\begin{proposition} \label{pr:duality for mf}
If $\bfwv$ is not of Calabi--Yau type,
then one has an equivalence
\begin{align}
\Funex \lb \emfzero{\bA^{n+3}}{K}{\bfV}, \perf \bfk \rb
\simeq \emf{\bA^{n+3}}{K}{\bfV}
\end{align}
of $\infty$-categories.
\end{proposition}

\begin{proof}
This follows from \pref{eq:Koszul duality}
and the fact that the right orthogonal to
$
\lb \cS_\bsi \rb_{\bsi \in I}
$
is zero since $\Gm \subset K$ is a dilating action
on the critical locus $\bA^1 \times \bszero$ of $\bfV$
with the origin as the unique fixed point.
\end{proof}

\pref{pr:duality for mf}
combined with
\cite[Theorem 6.11]{MR4442683}
proves
\cite[Conjecture 1.4]{MR4442683}
for non-log Calabi--Yau Brieskorn--Pham singularity.

\section{Rabinowitz Floer homology from mirror symmetry} \label{sc:RFH}

\subsection{}

We use the same notations as in \pref{sc:invertible}.
We regard $x_0$ as a section of the line bundle $\cL$ on $U$
associated with the character $\chi_0$ of $K$,
which in turn gives a natural transformation $s$
from the autoequivalence
$
\cL^{\vee} \otimes(-)
\simeq
(-\chi_0)
$
on
$
\emf{\bA^{n+3}}{K}{\bfW - x_0 \cdots x_{n+2}} \simeq \scoh U
$
to the identity functor.
The localization of $\scoh U$ along $s$ gives
$\scoh D \simeq \scoh U \setminus E$
since $E$ is defined by $x_0=0$
(cf.~\cite[Section 1]{MR2426130}).

Let $(d_1,\ldots,d_{n+2},h)$ be the sequence of positive integers
such that
$\gcd(d_1,\ldots,d_{n+2},h) = 1$ and
\begin{align}
\bfW(t^{d_1} x_1,\ldots,t^{d_{n+2}} x_{n+2})
= t^{h} \bfW(x_1,\ldots,x_{n+2}).
\end{align}
Then $h$ is the minimal positive integer
such that $x_0^h$
is invariant under the action of $\ker \chi \subset K$,
and one has
\begin{align}
\lb \cL^{\vee} \rb^{\otimes h} \otimes (-)
\simeq (- h \chi_0)
\simeq (- d_0 \chi)
\simeq [- 2 d_0]
\end{align}
where
\begin{align}
d_0 \coloneqq h - d_1 - \cdots - d_{n+2}.
\end{align}
One can regard $s^h$
as an element of $\HH^{2 d_0} \lb \scoh U \rb$,
so that
$
\scoh U
$
is linear over $\bfk[s^h]$.
The localization of $\scoh U$ along $s$
is equivalent to that along $s^h$;
\begin{align} \label{eq:scoh D from scoh U by inverting s}
\scoh D \simeq \scoh U \otimes_{\bfk[s^h]} \bfk[s^h,s^{-h}].
\end{align}
Assume \pref{cj:milnor},
so that
$
\scoh D \simeq \rfuk{\Uv}
$
by \pref{eq:stable hms}.
This implies
\begin{align}
\RFH^* \lb \Uv \rb
&\simeq \HH^* \lb \rfuk{\Uv} \rb \\
&\simeq \HH^* \lb \scoh U \otimes_{\bfk[s^h]} \bfk[s^h,s^{-h}] \rb \\
&\simeq \HH^* \lb \scoh U \rb \otimes_{\bfk[s^h]} \bfk[s^h,s^{-h}].
\end{align}

\subsection{}

Let $\bfWv$ be the transpose of $\bfW$
and $(\dv_1,\ldots,\dv_{n+2},\hv)$ be the sequence of positive integers
such that
$\gcd(\dv_1,\ldots,\dv_{n+2},\hv) = 1$ and
\begin{align}
\bfWv(t^{\dv_1} x_1,\ldots,t^{\dv_{n+2}} x_{n+2})
= t^{\hv} \bfWv(x_1,\ldots,x_{n+2}).
\end{align}
Then the link
\begin{align}
C \coloneqq \lc (x_1,\ldots,x_{n+2}) \in \bfWv^{-1}(0) \relmid
|x_1|^2+\cdots+|x_{n+2}|^2 = 1 \rc
\end{align}
of the singularity of $\bfWv^{-1}(0)$ at the origin
has an $S^1$-action defined by
\begin{align}
S^1 \ni t \colon
(x_1,\ldots,x_{n+2})
\mapsto \lb t^{\dv_1} x_1, \ldots, t^{\dv_{n+2}} x_{n+2} \rb,
\end{align}
which lifts to an $S^1$-action
\begin{align}
S^1 \ni t \colon
(x_0,\ldots,x_{n+2})
\mapsto \lb t^{\hv} x_0, t^{\dv_1} x_1, \ldots, t^{\dv_{n+2}} x_{n+2} \rb
\end{align}
on the total space
of the family
\begin{align} \label{eq:fUv}
\varphi \colon \fUv \coloneqq \lc (x_0,\ldots,x_{n+2}) \in \bC^{n+3} \relmid
\bfWv(x_1,\ldots,x_{n+2}) = x_0 \rc \to \bC^1,
\qquad (x_0,\ldots,x_{n+2}) \mapsto x_0.
\end{align}
Let $\mu$ be the endofunctor of $\wfuk{\Uv}$
defined as the clockwise monodromy of the family \pref{eq:fUv}
around the origin,
which is isomorphic to the composite
of inverse spherical twists along a distinguished basis
of vanishing cycles of $\bfWv$.
It is shown in \cite[Section 4.c]{MR1765826}
that
\begin{align}
  \mu^{\hv} \simeq \ld - 2 \dv_0 \rd
\end{align}
where
\begin{align}
\dv_0 \coloneqq \hv - \dv_1 - \cdots - \dv_{n+2}.
\end{align}
The wrapped Fukaya category
of the singular hypersurface $\bfWv^{-1}(0)$
in the sense of Auroux
(cf.~\cite[Definition 1]{MR4377932}),
defined as the localization of $\wfuk{\Uv}$
along the natural transformation
$\sv \colon \mu \to \id$
first introduced in \cite{MR2483942},
is equivalent to the quotient of $\wfuk{\Uv}$
by the split-closure
of the essential image of the cap functor
\cite[Corollary 1]{MR4377932}.
Since $\fuk{\Uv}$ is split-generated
by vanishing cycles,
which are in the essential image of the cap functor,
the wrapped Fukaya category
of the singular hypersurface $\bfWv^{-1}(0)$
is equivalent to the stable Fukaya category,
which in turn is equivalent to the Rabinowitz Fukaya category
$\rfuk{\Uv}$.

\subsection{}
Set
\begin{align}
 V \coloneqq \bfk x_0 \oplus \bfk x_1 \oplus \cdots \oplus \bfk x_{n+2}.
\end{align}
For $\gamma \in K$,
let $V_\gamma$ be the subspace of $\gamma$-invariant elements in $V$,
$S_\gamma$ be the symmetric algebra of $V_\gamma$,
$\bfV_\gamma$ be the restriction of $\bfV$ to $\Spec S_\gamma$,
and
$N_\gamma$ be the $K$-stable complement of $V_\gamma$ in $V$
so that
$V \cong V_\gamma \oplus N_\gamma$
as a $K$-module.
Then
\cite{MR2824483,MR3084707,MR3108698,MR3270588}
(cf.~also \cite[Theorem 3.1]{MR4442683})
shows that
$
 \HH^t \lb \emf{\bA^{n+2}}{K}{\bfV} \rb
$
is isomorphic to
\begin{multline} \label{eq:HHmf}
 \bigoplus_{\substack{\gamma \in \ker \chi, \ l \geq 0 \\ t - \dim N_\gamma = 2u }}
  \lb
  H^{-2l}(d \bfV_\gamma) \otimes \Lambda^{\dim N_\gamma} N_\gamma^\dual \rb_{(u+l)\chi} \\
  \oplus \bigoplus_{\substack{\gamma \in \ker \chi, \ l \geq 0 \\ t - \dim N_\gamma = 2u+1}}
  \lb
  H^{-2l-1}(d \bfV_\gamma) \otimes \Lambda^{\dim N_\gamma} N_\gamma^\vee \rb_{(u+l+1) \chi}.
\end{multline}
Here $H^i(d \bfV_\gamma)$ is the $i$-th cohomology of the Koszul complex
\begin{align} \label{eq:Koszul}
 C^*(d \bfV_\gamma) \coloneqq \lc
 \cdots
 \to \Lambda^2 V_\gamma^\dual \otimes S_\gamma (-2 \chi)
  \to V_\gamma^\dual \otimes S_\gamma (-\chi)
  \to S_\gamma \rc,
\end{align}
where the rightmost term $S_\gamma$ sits in cohomological degree 0, and
the differential is the contraction with
\begin{align}
 d \bfV_\gamma \in \lb V_\gamma \otimes S_\gamma \rb_{\chi}.
\end{align}

\subsection{} \label{sc:S^{n+1}}

As an example,
consider the case
$
\bfW = x_1^2 + \cdots + x_{n+2}^2
$,
where
$
\Uv \cong T^* S^{n+1}
$,
$
h=2
$,
$
d_0=-n
$,
and
$
\HH^* \lb \emf{\bA^{5}}{K}{\bfW} \rb \otimes_{\bfk[s^h]} \bfk[s^h,s^{-h}]
$
has a basis consisting of
\begin{itemize}
 \item
$x_0^{2m}$ for $m \in \bZ$ and $\gamma=(1,\ldots,1)$
of degree $-2mn$,
 \item
$x_0^{2m+1} \otimes x_0^\vee$ for $m \in \bZ$ and $\gamma=(1,\ldots,1)$
of degree $-2mn+1$
\end{itemize}
and,
if $n$ is even,
in addition to the above,
\begin{itemize}
 \item
$x_0^{2m} \otimes x_1^\vee \otimes \cdots x_{n+2}^\vee$ for $m \in \bZ$
and $\gamma = (1,-1,-1,\ldots,-1)$
of degree $(2m+1)n$,
 \item
$x_0^{2m+1} \otimes x_0^\vee \otimes \cdots x_{n+2}^\vee$ for $m \in \bZ$
and $\gamma = (1,-1,-1,\ldots,-1)$
for $m \in \bZ$ of
degree $(2m+1)n+1$.
\end{itemize}

\subsection{} \label{sc:T^* S^1}

Let $U_n$ be the Liouville domain $\bCx \cong T^* S^1$
equipped with the grading
determined by the quadratic differential
$x^n (d \log x)^{\otimes 2}$.
Then $\wfuk{U_n}$ is generated
by the cotangent fiber,
whose endomorphism $A_\infty$-algebra
is isomorphic to the free algebra $\bfk \la u, u^{-1} \ra$
generated by an element $u$ of degree $n$
and its inverse $u^{-1}$,
so that
\begin{align}
\wfuk{U_n}
&\simeq \module \bfk \la u, u^{-1} \ra \\
&\simeq \rfuk{T^* S^{n+1}},
\end{align}
and hence
\begin{align}
\SH^*(U_n)
&\simeq \HH^*(\bfk \la u, u^{-1} \ra) \\
&\simeq \RFH(T^* S^{n+1}).
\end{align}
Reeb orbits in the contact boundary of $U_n$
with winding number $w$
come in a family parametrized by $S^1$,
and
a Bott--Morse model
of the symplectic cohomology
gives
two generators
$p_w$ and $q_w$
of degrees $n w$ and $n w + 1$
in such a way that
the Floer differential is given by
\begin{align}
d p_w = (1 - (-1)^{n w}) q_w
\end{align}
(see \cite[Propostion 3.9]{MR2475400}).
It follows that
\begin{itemize}
\item 
if $n$ is even,
then $p_w$ and $q_w$ for all $w \in \bZ$ survives,
\item
if $n$ is odd,
then $p_w$ and $q_w$ for odd $w$ annihilates each other,
and only $p_w$ and $q_w$ for even $w$ survives,
\end{itemize}
and one can identify
\begin{align}
p_{2m}
&=
x_0^{2m}, \\
q_{2m}
&=
x_0^{2m+1} \otimes x_0^\vee, \\
p_{2m+1}
&=
x_0^m \otimes x_1^\vee \otimes \cdots \otimes x_{n+2}^\vee, \\
q_{2m+1}
&=
x_0^{2m+1} x_0^\vee \otimes x_1^\vee \otimes \cdots \otimes x_{n+2}^\vee.
\end{align}

\subsection{}

One can also compute
$
\HH^* \lb \bfk \la u, u^{-1} \ra \rb
$
in a purely algebraic way.
The enveloping algebra
$
R^{\op} \otimes R
$
of
$
R \coloneqq \bfk \la u, u^{-1} \ra
$
is isomorphic to
$
E = \bfk \la \lambda^{\pm 1}, \rho^{\pm 1} \ra/( \lambda \rho -(-1)^n \rho \lambda)
$.
The diagonal bimodule $\Delta$
regarded as a right $E$-module
is the $\bfk$-vector space $\bigoplus_{i=-\infty}^\infty \bfk u^i$
equipped with the action
\begin{align}
u^i \cdot \lambda &= (-1)^{in} u^{i+1}, \\
u^i \cdot \rho &= u^{i+1}.
\end{align}
One has
\begin{align}
\Delta \simeq \lc E[-n] \xto{(\lambda-\rho) \cdot } E \rc,
\end{align}
where the right terms sits in degree zero,
so that
the Hochschild complex of $R$ is given by
\begin{align}
\hom_E(\Delta, \Delta)
&\simeq \hom_E \lb \lc E[-n] \xto{(\lambda-\rho) \cdot} E \rc, \Delta \rb \\
&\simeq \Delta \otimes_E \lc E \xto{\cdot (\lambda-\rho)} E[n] \rc \\
&\simeq \lc \Delta \xto{\cdot (\lambda-\rho)} \Delta[n] \rc, \label{eq:alg}
\end{align}
where the left term sits in degree zero.
Since
\begin{align}
u^i \cdot (\lambda-\rho) =
\begin{cases}
  0 & i \text{ is even}, \\
  (1-(-1)^{in}) u^{i+1} & i \text{ is odd},
\end{cases}
\end{align}
the complex \pref{eq:alg} is identical
to the Floer complex appearing in \pref{sc:T^* S^1}.

\subsection{}

As another example,
consider the case
\begin{align} \label{eq:BS_k,m-k,2,2}
\bfW(x_1,x_2,x_3,x_4) = x_1^k + x_2^{m-k} + x_3^2 + x_4^2.
\end{align}
The hypersurface $\bfW^{-1}(0)$ has an isolated cDV singularity
at the origin.

\subsection{}

Recall from \cite[Theorem 1.1]{MR715649}
that a 3-fold singularity is terminal of index 1
if and only if it is an isolated cDV singularity.
It follows from \cite[Theorem 1.1]{MR3489704}
that the link of a terminal singularity
is index-positive
in the sense of \cite[Section 9.5]{MR3797062},
so that the Rabinowitz Floer cohomology is an invariant of the link
(i.e., does not depend on the symplectic filling)
(see \cite[Proposition 9.17]{MR3797062}).

\subsection{}

One can find a formal change of coordinates
transforming $\bfV(x_0,x_1,x_2,x_3,x_4)$ to $\bfW(x_1,x_2,x_3,x_4)$
to show
$
\emf{\bA^{5}}{K}{\bfV}
\simeq
\emf{\bA^{5}}{K}{\bfW}
$
just as in \pref{sc:invertible}.
Explicit computations of
\linebreak
$
\dim \HH^t \lb \emf{\bA^{5}}{K}{\bfW} \rb
$
are discussed
for simple singularities in
\cite[Section 5]{MR4371540},
for more general cases in \cite[Section 3.1]{MR4648096},
and for all of \pref{eq:BS_k,m-k,2,2} in \cite[Theorem C]{2404.17301}.

\subsection{}

One has
$
\HH^t \lb \emf{\bA^{5}}{K}{\bfW} \rb = 0
$
for $t > 3$,
and a basis of
$
\HH^3 \lb \emf{\bA^{5}}{K}{\bfW} \rb
$
consists of
\begin{align}
x_0^{\vee} x_1^{\vee} x_2^{\vee} x_3^{\vee} x_4^{\vee}
\in H^0(C^*(d \bfW_\gamma)) \otimes \Lambda^5 N_\gamma^{\vee}
\cong \Lambda^5 N_\gamma^{\vee}
\end{align}
in the direct summand of \pref{eq:HHmf}
such that
$
V_\gamma = 0
$
and
\begin{align}
x_0^{\vee} x_1^{\vee} x_2^{\vee} x_3^{\vee} x_4^{\vee}
\in H^{-1}(C^*(d \bfW_\gamma)) \otimes \Lambda^4 N_\gamma^{\vee}
\cong \bfk[x_0] \otimes (\bfk x_0)^\dual \otimes \Lambda^4 N_\gamma^{\vee}
\end{align}
in the direct summand of \pref{eq:HHmf}
such that
$
V_\gamma = \bfk x_0
$.
It follows that
$
\dim \HH^3\lb \emf{\bA^{5}}{K}{\bfW} \rb
$
is equal to the number
$
(m-k-1)(k-1)
$
of
\begin{align}
\gamma \in \ker \chi
= \{
(t_0,t_1,t_2,t_3,t_4) \in (\bfk^\times)^5
\mid
t_1^k = t_2^{m-k} = t_3^2 = t_4^2 = t_0 t_1 t_2 t_3 t_4 = 1
\}
\end{align}
such that
$t_i \ne 1$ for $i=1,2,3,4$,
which in turn is equal
to the Milnor number of the singularity
defined by \pref{eq:BS_k,m-k,2,2}.
One also has
$
\HH^2 \lb \emf{\bA^{5}}{K}{\bfW} \rb = 0
$
and
$
\dim \HH^{2i} \lb \emf{\bA^{5}}{K}{\bfW} \rb
=
\dim \HH^{2i+1} \lb \emf{\bA^{5}}{K}{\bfW} \rb
$
for all $i \le 0$.
A basis of
$
\HH^{-2i} \lb \emf{\bA^{5}}{K}{\bfW} \rb
$
for $i \in \bN$
can be divided into two kinds.

\subsection{}

Let
$
m_1 \colon \bN \to \bN
$
and
$
b_1 \colon \bN \to \{ 0, \ldots, k-1 \}
$
(resp.~%
$
m_2 \colon \bN \to \bN
$
and
$
b_2 \colon \bN \to \{ 0, \ldots, m-k-1 \}
$%
)
be the quotient and the remainder
by $k$ (resp.~$m-k$),
so that
\begin{align} \label{eq:b1b2m1m2_as_func_of_b0}
b_0 = b_1(b_0) + k m_1(b_0) = b_2(b_0) + (m-k) m_2(b_0)
\end{align}
for any $b_0 \in \bN$.
A basis
of the first kind
in cohomological degree $- 2i$
for $i \in \bN$
is parametrized by the set
\begin{align} \label{eq:I}
\rI_{k,m-k}^i
\coloneqq
\lc
b_0 \in \bN
\relmid
b_1(b_0) \ne k-1, \ 
b_2(b_0) \ne m-k-1, \ 
m_1(b_0)+m_2(b_0)=i
\rc
\end{align}
as
\begin{align}
\begin{cases}
x_0^{b_0} x_1^{b_1(b_0)} x_2^{b_2(b_0)} & b_0 \text{ is even and } \gamma = (1,1,1,1,1), \\
x_0^{b_0} x_1^{b_1(b_0)} x_2^{b_2(b_0)} x_3^{\vee} x_4^{\vee} & b_0 \text{ is odd and } \gamma = (1,1,1,-1,-1).
\end{cases}
\end{align}

\subsection{}

A basis of the second kind
in cohomological degree $- 2i$
is parametrized by the product set of
\begin{align} \label{eq:II}
\rII_{k,m-k}^i
\coloneqq
\lc
b_0 \in \bN
\relmid
\exists n_1, n_2 \in \bN,
n_1+n_2-1=i \text{ and }
b_0 = -1 + k n_1 = -1 + (m-k) n_2
\rc
\end{align}
and
\begin{align}
(\bmu_k \times \bmu_{m-k}) \setminus \{ 1 \}
=
\lc
\xi \in \bfk
\relmid
\xi \ne 1 \text{ and }
\xi^k = \xi^{m-k} = 1
\rc
\end{align}
as
\begin{align}
\begin{cases}
x_0^{b_0} x_1^{\vee} x_2^{\vee} & b_0 \text{ is even and } \gamma = (1,\xi,\xi^{-1},1,1),\\
x_0^{b_0} x_1^{\vee} x_2^{\vee} x_3^{\vee} x_4^{\vee} & b_0 \text{ is odd and } \gamma = (1,\xi,\xi^{-1},-1,-1).
\end{cases}
\end{align}

\subsection{}

\pref{tb:dim HH}
shows
$
\dim \HH^t ( \operatorname{mf} ([\bA^{5}/K], \bfW) )
$
for $3 \ge t \ge -15$
for small $(m,k)$
to give an idea about the outcome of this computation.

\begin{table}[htbp]
\begin{align*}
\begin{array}{*{20}{c}}
\toprule
(m,k) &3 & 2 & 1 & 0 & -1 & -2 & -3 & -4 & -5 & -6 & -7 & -8 & -9 & -10 & -11 & -12 & -13 & -14 & -15 \\
\midrule
(4,2)& 1& 0 & 1 & 1 & 1 & 1 & 1 & 1 & 1 & 1 & 1 & 1 & 1&1 &1 &1 &1 &1 &1 \\
(5,2)& 2 & 0 & 1 &1 &0 &0 &0 &0 &1 &1 &0 &0 &1 &1 &0 &0 &0 &0 &1 \\
(6,2)&3 & 0 &1 &1 &1 & 1& 1&1 &1 &1 &1 &1 &1 &1 &1 &1 &1 &1 &1 \\
(7,2)&4 & 0 &1  &1 &1 &1 &0 &0 &0 &0 &1 &1 &1 &1 &0 &0 &1 &1 &1 \\
(6,3) &4 &0 &2 &2 &2 &2 &2 &2 &2 &2 &2 &2 &2 &2 &2 & 2&2 &2 &2 \\
(7,3) &6 &0 &2 &2 &0 &0 &1 &1 &1 &1 &0 &0 &2 &2 &0 &0 &2 &2 &0 \\
(8,3) & 8 &0 &2 &2 &1 &1 &0 &0 &2 &2 &0 &0 &1 & 1&2 &2 &0 &0 &2 \\
(8,4) &9 &0 &3 &3 &3 &3 &3 &3 &3 &3 &3 &3 &3 &3 &3 &3 &3 &3 &3 \\
\bottomrule
\end{array}
\end{align*}
\caption{Examples of $
\dim \HH^t \lb \emf{\bA^{5}}{K}{\bfW} \rb
$}
\label{tb:dim HH}
\end{table}

\subsection{}

Define
\begin{align}
h_{k,m-k}^i
\coloneqq
\dim \HH^{-2i} \lb \emf{\bA^{5}}{K}{\bfW} \rb.
\end{align}
Set $\ell = \gcd(k,m-k)$ and
write $(k, m-k) = (e \ell, f \ell)$,
so that $\gcd(e,f)=1$.

\begin{proposition} \label{pr:recursion}
One has
\begin{align}
h_{e \ell, f \ell}^i
= \ell h_{e,f}^i + \ell - 1.
\end{align}
\end{proposition}

\begin{proof}
If $i \not \equiv -1 \mod e+f$,
then one has
$\rII_{e \ell, f \ell}^i = \emptyset$
so that
\begin{align}
h_{e \ell, f \ell}^i = |\rI_{e \ell, f \ell}^i|.
\end{align}
If we write
\begin{align}
\rI_{e,f}^i
=
\{
b_0^{\min}, b_0^{\min}+1, \ldots, b_0^{\max}
\},
\end{align}
then one has
\begin{align}
\rI_{\ell e,\ell f}^i
=
\{
\ell b_0^{\min}, \ell b_0^{\min}+1, \ldots, \ell b_0^{\max} + 2 \ell - 2
\},
\end{align}
so that
\begin{align}
|\rI_{e \ell, f \ell}^i|
&= (\ell b_0^{\max} + 2 \ell - 2) - \ell b_0^{\min} + 1 \\
&= \ell (b_0^{\max} - b_0^{\min} + 1) + \ell - 1 \\
&= \ell |I_{e,f}^i| + \ell - 1.
\end{align}

If $i \equiv -1 \mod e+f$,
then
one has $\rI_{e \ell, f \ell}^i = \emptyset$,
$|\rII_{e \ell, f \ell}^i|=1$,
and
\begin{align}
h_{e \ell, f \ell}^i
= \left| \lb \bmu_{e \ell} \cap \bmu_{f \ell} \rb \setminus \{ 1 \} \right|
= \ell - 1,
\end{align}
which is equal to $\ell h_{e,f}^i + \ell - 1$
since $h_{e,f}^i=0$.
\end{proof}

\subsection{}

The following conjecture is proposed in \cite{MR4648096}
(see also \cite{2405.03475}):

\begin{conjecture}[{\cite[Conjecture 1.4]{MR4648096}}] \label{cj:Evans-Lekili}
A compound Du Val singularity
admits a small resolution
if and only if the dimension of the symplectic cohomology of its Milnor fiber
is constant in every negative cohomological degree.
Furthermore,
if this is the case,
then this dimension is equal to the number of irreducible components
of the exceptional locus of a small resolution.
\end{conjecture}

The following refinement of \pref{cj:Evans-Lekili}
was also proposed by the authors of \cite{MR4648096}
based on unpublished calculations:

\begin{conjecture} \label{cj:refined_Evans-Lekili}
Let $Y \to X$ be a small $\bQ$-factorialization
of a compound Du Val singularity
$P \in X$.
Let further $r$ be the number of irreducible components
of the exceptional locus,
$\Uv$ be the Milnor fiber of $P \in X$,
and
$\Uv_1,\ldots,\Uv_s$ be the Milnor fibers
of the resulting $\bQ$-factorial singularities
$Q_1,\ldots, Q_s \in Y$.
Then one has
\begin{align}
\dim \SH^i \lb \Uv \rb = \sum_{j=1}^s \dim \SH^i \lb \Uv_j \rb + r
\end{align}
for any $i < 0$.
\end{conjecture}

\subsection{} \label{sc:blow-up}

The blow-up of
\begin{align}
  \Spec \bfk[x,y,z,w]/(x y - f(z,w) g(z,w))
\end{align}
along the ideal
$(x,f(z,w))$
is defined by
$
x v = f(z,w) u
$
and
$
y u = g(z,w) v
$
inside
$
\bA^4_{x,y,z,w} \times \bP^1_{u:v}
$,
which is the union of
\begin{align}
  \Spec \bfk[x,z,w,v] / (x v - f(z,w))
\end{align}
and
\begin{align}
  \Spec \bfk[y,z,w,u]/(y u - g(z,w)).
\end{align}

\subsection{}
Now we discuss \pref{cj:refined_Evans-Lekili}
for the case
when $X$ is defined by
$\bfWv$ given by \pref{eq:BS_k,m-k,2,2}
(and hence the transpose $\bfW$ of $\bfWv$ is also given by \pref{eq:BS_k,m-k,2,2}).
By starting from the 3-fold
\begin{align}
X = \left. \Spec \bfk[x,y,z,w] \middle/ \lb x^2 + y^2 + z^{e \ell} + w^{f \ell} \rb \right.
\cong
\left. \Spec \bfk[x,y,z,w] \middle/ \lb x y - \prod_{i=0}^{\ell-1} \lb z^e - \zeta_{\ell}^i w^f \rb \rb \right.
\end{align}
where
$
\zeta_\ell \coloneqq \exp \lb 2 \pi \sqrt{-1}/\ell \rb
$
and performing the blown-up in \pref{sc:blow-up} $\ell-1$ times,
one obtains a chain of $\ell-1$ exceptional $\bP^1$'s and
$\ell$ copies of $\bQ$-factorial compound Du Val singularities
isomorphic to $x^2+y^2+z^e+w^f$.
It follows that one has $r = \ell - 1$, $s = \ell$,
and $\Uv_1, \ldots, \Uv_s$ are Milnor fibers of
$x^2+y^2+z^e+w^f$ in this case.
Since \pref{cj:milnor} holds for \pref{eq:BS_k,m-k,2,2}
by \pref{th:hms for U},
one has $\dim \SH^i(\Uv) = h_{e \ell, f \ell}^i$
and $\dim \SH^i(\Uv_j) = h_{e,f}^i$ for $j = 1, \ldots, s$.
Now
\pref{pr:recursion} shows that
\pref{cj:refined_Evans-Lekili} holds
for the singularity
defined by \pref{eq:BS_k,m-k,2,2}.

\subsection{}

For $\bfW$ given by \pref{eq:BS_k,m-k,2,2},
one has
\begin{align}
(h, d_0) =
\begin{cases}
(e f \ell/2, -e-f) & \ell \text{ is even}, \\
(e f \ell, -2e-2f) & \ell \text{ is odd}.
\end{cases}
\end{align}
The isomorphism
$
\HH^* \lb \scoh D \rb
\simeq
\HH^* \lb \scoh U \rb \otimes_{\bfk[s^h]} \bfk[s^h,s^{-h}]
$
and the explicit description of\linebreak
$
\HH^* \lb \scoh U \rb
\cong
\HH^* \lb \emf{\bA^{5}}{K}{\bfW} \rb
$
above imply
$
\HH^i(\scoh D) \cong \HH^i(\scoh U)
$
for all $i \le 0$
and that
the isomorphisms
$
s^h \colon \HH^i(\scoh U) \simto \HH^{i+2d_0}(\scoh U)
$
for $i \le 0$
extend to isomorphisms
$
s^h \colon \HH^i(\scoh D) \simto \HH^{i+2d_0}(\scoh D)
$
for all $i \in \bZ$.
It follows that
the symplectic cohomology and the Rabinowitz Floer cohomology
of the mirror
$\Uv$
satisfies
$
\RFH^i \lb \Uv \rb \cong \SH^i \lb \Uv \rb
$
for all $i \le 0$
and
$
\RFH^i \lb \Uv \rb \cong \RFH^{i+2d_0} \lb \Uv \rb
$
for all $i \in \bZ$.

\subsection{}

As an application of the non-vanishing of the Rabinowitz Floer homology,
one can obtain
the non-displaceability of the contact boundary of $\Uv$
by \cite[Theorem 1.2]{MR2461235}.

\bibliographystyle{amsalpha}
\bibliography{bibs}

\end{document}